\documentclass[a4paper,12pt,intlimits,oneside]{amsart}

\usepackage{enumerate}
\usepackage{amsfonts,amsmath}

\usepackage{latexsym,amssymb}

\usepackage{mathrsfs}

\newcommand{\comment}[1]{}

\newcommand{\R}{{\mathbb R}}
\newcommand{\X}{{\mathcal X}}

\def\supp{{\mbox{\small\rm  supp}}\,}

\begin{document}

\title[On weak$^*$-convergence in $H^1(\X)$]{On weak$^*$-convergence in the Hardy space $H^1$ over spaces of homogeneous type}         

\author{Ha Duy Hung and Luong Dang Ky $^*$}

\thanks{ 	
	$^{*}$Corresponding author}

\begin{abstract}
	Let $\X$ be a complete space of homogeneous type. In this note, we prove that the weak$^*$-convergence is true in the Hardy space $H^1(\X)$ of Coifman and Weiss.
\end{abstract}

\maketitle
\newtheorem{theorem}{Theorem}[section]
\newtheorem{lemma}{Lemma}[section]
\newtheorem{proposition}{Proposition}[section]
\newtheorem{remark}{Remark}[section]
\newtheorem{corollary}{Corollary}[section]
\newtheorem{definition}{Definition}[section]
\newtheorem{example}{Example}[section]
\numberwithin{equation}{section}
\newtheorem{Theorem}{Theorem}[section]
\newtheorem{Lemma}{Lemma}[section]
\newtheorem{Proposition}{Proposition}[section]
\newtheorem{Remark}{Remark}[section]
\newtheorem{Corollary}{Corollary}[section]
\newtheorem{Definition}{Definition}[section]
\newtheorem{Example}{Example}[section]
\newtheorem*{theoremjj}{Theorem J-J}
\newtheorem*{theorema}{Theorem A}
\newtheorem*{theoremb}{Theorem B}
\newtheorem*{theoremc}{Theorem C}

\section{Introduction}

Let $d$ be a quasi-metric on a set $\X$, that is, $d$ is a nonnegative function on $\mathcal X\times \mathcal X$ satisfying
\begin{enumerate}[\hskip 0.5cm $\bullet$]
	\item $d(x,y)=d(y,x)$,
	\item $d(x,y)>0$ if and only if $x\ne y$,
	\item there exists a constant $\kappa_1\geq 1$ such that for all $x,y,z\in \mathcal X$,
	\begin{equation}\label{the metric constant}
		d(x,z)\leq \kappa_1(d(x,y)+ d(y,z)).
	\end{equation}
\end{enumerate}
A trip $(\mathcal X, d,\mu)$ is called a {\sl space of homogeneous type} in the sense of Coifman and Weiss if $\mu$ is a regular Borel measure satisfying {\sl doubling property}, i.e. there exists a constant $\kappa_2 >1$ such that for all $x\in \mathcal X$ and $r>0$,
\begin{equation}\label{the doubling constant}
\mu(B(x,2r))\leq \kappa_2 \mu(B(x,r)).
\end{equation}

In this paper, we always assume that $(\mathcal X, d,\mu)$ is a complete space of homogeneous type, $\mu(\X)=\infty$ and $0<\mu(B)<\infty$ for any ball $B\subset \X$.

Recall (see \cite{CW}) that a function $a$ is called an {\sl $H^1$-atom} related to the ball $B\subset \X$ if
\begin{enumerate}[\hskip 0.5cm $\bullet$]
\item $\supp a\subset B$;
	
\item  $\|a\|_{L^\infty(\X)}\leq \mu(B)^{-1}$;
	
\item $\int_{\X} a(x) d\mu(x)=0$.
\end{enumerate}
The Hardy space $H^1(\X)$ is defined as the set of all $f=\sum_{j=1}^\infty \lambda_j a_j$ with $\{a_j\}_{j=1}^\infty$ are $H^1$-atoms and $\{\lambda_j\}_{j=1}^\infty \subset \mathbb C$ is such that $\sum_{j=1}^\infty |\lambda_j|<\infty$. The norm on $H^1(\X)$ is then defined by
$$\|f\|_{H^1(\X)}:= \inf\left\{\sum_{j=1}^\infty |\lambda_j|: f=\sum_{j=1}^\infty \lambda_j a_j\right\}.$$

It is well-known (see \cite{CW}) that the dual space of $H^1(\X)$ is $BMO(\X)$ the space of all  locally integrable functions $f$ with
$$\|f\|_{BMO(\X)}:=\sup\limits_{B}\frac{1}{|B|}\int_B \Big|f(x)-\frac{1}{|B|}\int_B f(y) d\mu(y)\Big|d\mu(x)<\infty,$$
where the supremum is taken over all balls $B\subset \X$. Furthermore,  $H^1(\X)$  is itself the dual space of $VMO(\X)$ the closure in $BMO(\X)$ norm of the set $C_c(\X)$ of all continuous functions with compact support.

 The aim of the present paper is to establish the following.

\begin{theorem}\label{the main theorem}
Suppose that $\{f_n\}_{n=1}^\infty$ is a bounded sequence in $H^1(\X)$, and that $\lim_{n\to\infty} f_n(x) = f(x)$ for almost every $x\in\X$. Then, $f\in H^1(\X)$ and $\{f_n\}_{n= 1}^\infty$ weak$^*$-converges to $f$, that is, for every $\varphi\in VMO(\X)$, we have
$$\lim_{n\to\infty} \int_{\X} f_n(x) \varphi(x)d\mu(x) = \int_{\X} f(x) \varphi(x) d\mu(x).$$
\end{theorem}

It should be pointed out that, when $\X=\R^d$, Theorem \ref{the main theorem} was proved firstly by Jones and Journ\'e \cite{JJ} when trying to answer a question of Lions and Meyer \cite{CLMS}. When $\X$ is a norm space of homogeneous type, Theorem \ref{the main theorem} was established by Grafakos and Rochberg \cite{GR}.


\section{Proof of Theorem \ref{the main theorem}}


Let $M$ be the classical Hardy-Littlewood maximal function. The following result is well-known (see \cite{CW}).

\begin{lemma}\label{the weak type (1,1) for the Hardy-Littlewood function}
There exists a constant $C_1>0$ such that
$$\mu(\{x\in \X: Mg(x)>\lambda\})\leq C_1 \frac{1}{\lambda} \int_{\X} |g(x)| d\mu(x)$$
for all $g\in L^1(\X)$ and $\lambda>0$.
\end{lemma}

By a standard argument (cf. \cite{CR, GR}), we get the following lemma.

\begin{lemma}\label{a lemma of Coifman and Rochberg}
There exists a constant $C_2>0$ such that, for any Borel set $E$,
$$\|\log(M(\chi_E))\|_{BMO(\X)}\leq C_2.$$
\end{lemma}

\begin{proof}[Proof of Theorem \ref{the main theorem}]
 Without loss of generality, we can assume that $\|f_n\|_{H^1}\leq 1$ for all $n\geq 1$. Since the Fatou's lemma, we see that $\|f\|_{L^1(\X)}\leq 1$. By this and a standard function theoretic argument, it suffices to show that
\begin{equation}\label{the main theorem, 1}
\lim_{n\to\infty}\int_{\X} f_n(x) \varphi(x) d\mu(x) = \int_{\X} f(x) \varphi(x) d\mu(x)
\end{equation}
for all $\varphi\in C_c(\X)$ with $\|\varphi\|_{L^1(\X)}, \|\varphi\|_{L^\infty(\X)}\leq 1$.

Fix $\varepsilon >0$. As $f\in L^1(\X)$, there exists a positive number $\beta$ such that $\int_{E}|f|d\mu<\varepsilon$ for any Borel set $A$ satisfying $\mu(A)<\beta$. By $\varphi\in C_c(\X)$ and \cite[Theorems 2 and 3]{MS}, there exists $\alpha\in (0, \frac{\beta}{C_1 \varepsilon})$, where the constant $C_1$ is as in Lemma \ref{the weak type (1,1) for the Hardy-Littlewood function}, such that if $B$ is a ball in $\X$ satisfying $\mu(B)<\alpha$, then
\begin{equation}\label{Macias, Segovia and Ky}
|\varphi(x) - \varphi(y)|<\varepsilon
\end{equation}
for all $x,y\in B$. On the other hand, by the Egorov's theorem and the regularity of $\mu$, there exists an open set $E\subset\X$ such that $\mu(E)< \alpha \varepsilon e^{-\frac{1}{\varepsilon}}$ and $f_n\to f$ uniformly on  $\supp\varphi\setminus E$. Define $\tau:= \max\{0, 1+ \varepsilon \log(M(\chi_{E}))\}$. It is clear that $0\leq \tau\leq 1$ and $\tau(x)=1$ for all $x\in E$. We now claim that
\begin{equation}\label{BMO-estimate of Jones and Journe}
\|\varphi\tau\|_{BMO(\X)}\leq (2+ 2C_1+ 3C_2)\varepsilon,
\end{equation}
where the constant $C_2$ is as in Lemma \ref{a lemma of Coifman and Rochberg}.

Assume that (\ref{BMO-estimate of Jones and Journe}) holds for a moment. Since $f_n\to f$ uniformly on  $\supp\varphi\setminus E$, there exists $n_0\in\mathbb N$ such that $\|f_n-f\|_{L^{\infty}(\supp \varphi\setminus E)}<\varepsilon$ for all $n\geq n_0$. Applying Lemma \ref{the weak type (1,1) for the Hardy-Littlewood function} with $g=\chi_E$ and $\lambda= e^{-\frac{1}{\varepsilon}}$, we obtain that 
$$\mu(\supp \tau)\leq C_1 \mu(E) e^{\frac{1}{\varepsilon}}< C_1 \alpha \varepsilon <\beta.$$
 Therefore, for any $n\geq n_0$,
\begin{eqnarray*}
&&\left|\int_{\X} (f_n-f)\varphi d\mu\right| \\
&\leq& \left|\int_{\supp \varphi\setminus E} (f_n -f) \varphi(1-\tau) d\mu\right|  + \left|\int_{\supp \tau} f \varphi \tau d\mu\right| + \left|\int_{\X} f_n \varphi \tau d\mu\right|\\
&\leq& \|f_n-f\|_{L^{\infty}(\supp \varphi\setminus E)} \|\varphi\|_{L^1(\X)} + \|\varphi\|_{L^\infty(\X)}\int_{\supp\tau} |f| d\mu + \|f_n\|_{H^1(\X)}\|\varphi\tau\|_{BMO(\X)}\\
&\leq& \varepsilon + \varepsilon + (2+ 2C_1 + 3 C_2)\varepsilon= (4+ 2C_1 + 3 C_2) \varepsilon.
\end{eqnarray*}
This proves that (\ref{the main theorem, 1}) holds.

Let us now show (\ref{BMO-estimate of Jones and Journe}). Let $B$ be an arbitrary ball in $\X$. If $\mu(B)\geq \alpha$, then
\begin{eqnarray*}
\frac{1}{\mu(B)}\int_{B}\left|\varphi(x) \tau(x) -\frac{1}{\mu(B)}\int_B \varphi(y)\tau(y)d\mu(y)\right|d\mu(x) &\leq& 2 \frac{1}{\mu(B)} \int_B |\varphi\tau| d\mu\\
&\leq& 2 \frac{1}{\alpha}\|\varphi\|_{L^\infty(\X)} \mu(\supp \tau)\\
&\leq& 2 C_1 \varepsilon.
\end{eqnarray*}
Otherwise, by (\ref{Macias, Segovia and Ky}) and Lemma \ref{a lemma of Coifman and Rochberg}, we have
\begin{eqnarray*}
	&&\frac{1}{\mu(B)}\int_{B}\left|\varphi(x) \tau(x) -\frac{1}{\mu(B)}\int_B \varphi(y)\tau(y)d\mu(y)\right|d\mu(x)\\
	&\leq& 2 \frac{1}{\mu(B)}\int_{B}\left| \varphi(x) \tau(x) -\frac{1}{\mu(B)}\int_B \varphi(y) d\mu(y) \frac{1}{\mu(B)}\int_B \tau(y) d\mu(y) \right|d\mu(x) \\
	&\leq& 2\|\tau\|_{L^\infty(\X)}\sup_{x,y\in B} |\varphi(x) -\varphi(y)| + 2 \|\varphi\|_{L^\infty(\X)} \|\tau\|_{BMO(\X)}\\
	&\leq& 2 \varepsilon + 2 \frac{3}{2} \|\varepsilon \log(M(\chi_E))\|_{BMO(\X)}\leq (2+ 3 C_2)\varepsilon,
\end{eqnarray*}
which proves that (\ref{BMO-estimate of Jones and Journe}) holds.

\end{proof}

{\bf Acknowledgements.} The paper was completed when the second author was visiting
to Vietnam Institute for Advanced Study in Mathematics (VIASM). He would like
to thank the VIASM for financial support and hospitality.

\bigskip

\noindent Ha Duy Hung

\medskip

\noindent High School for Gifted Students,
Hanoi National University of Education, 136 Xuan Thuy, Hanoi, Vietnam

\smallskip

\noindent {\it E-mail:} \texttt{hunghaduy@gmail.com}

\bigskip

\noindent Luong Dang Ky

\medskip

\noindent Department of Mathematics,
University of Quy Nhon,
170 An Duong Vuong,
Quy Nhon, Binh Dinh, Vietnam
\smallskip

\noindent {\it E-mail}: \texttt{luongdangky@qnu.edu.vn}

\end{document}